\documentclass{amsart}

\usepackage{graphicx}
\usepackage{pstricks}
\usepackage{amscd}
\usepackage{amsmath}
\usepackage{amsxtra}

\usepackage{amsfonts,amssymb,amsmath,amsthm}
\usepackage{url}
\usepackage{enumerate}

\newcommand{\strutstretchdef}{\newcommand{\strutstretch}{\vphantom{\raisebox{1pt}{$\big($}\raisebox{-1pt}{$\big($}}}}
\theoremstyle{plain}
\newtheorem{theorem}{Theorem}[section]
\newtheorem{lemma}[theorem]{Lemma}

\newtheorem{corollary}[theorem]{Corollary}

\theoremstyle{definition}
\newtheorem{definition}[theorem]{Definition}

\theoremstyle{remark}
\newtheorem{remark}[theorem]{Remark}

\numberwithin{equation}{section}

\newlength{\struh}
\newlength{\textminustop}
\strutstretchdef
\usepackage{mathrsfs}
\usepackage{epsfig}
\usepackage{ifthen}
\usepackage[all]{xy}

\newcommand{\ncom}{\newcommand}
\ncom{\bq}{\begin{equation}}
\ncom{\eq}{\end{equation}}
\ncom{\beqn}{\begin{eqnarray*}}
\ncom{\eeqn}{\end{eqnarray*}}
\ncom{\beq}{\begin{eqnarray}}
\ncom{\eeq}{\end{eqnarray}}
\ncom{\nno}{\nonumber}
\ncom{\rar}{\rightarrow}
\ncom{\Rar}{\Rightarrow}
\ncom{\noin}{\noindent}
\ncom{\bc}{\begin{centre}}
\ncom{\ec}{\end{centre}}
\ncom{\sz}{\scriptsize}
\ncom{\rf}{\ref}
\ncom{\sgm}{\sigma}
\ncom{\Sgm}{\Sigma}
\ncom{\dt}{\delta}
\ncom{\Dt}{Delta}
\ncom{\lmd}{\lambda}
\ncom{\Lmd}{\Lambda}
\ncom{\eps}{\epsilon}
\ncom{\pcc}{\stackrel{P}{>}}
\ncom{\dist}{{\rm\,dist}}
\ncom{\sspan}{{\rm\,span}}
\ncom{\im}{{\rm Im\,}}
\ncom{\sgn}{{\rm sgn\,}}
\ncom{\ba}{\begin{array}}
\ncom{\ea}{\end{array}}
\ncom{\eop}{\hfill{{\rule{2.5mm}{2.5mm}}}}
\ncom{\eoe}{\hfill{{\rule{1.5mm}{1.5mm}}}}
\ncom{\eof}{\hfill{{\rule{1.5mm}{1.5mm}}}}
\ncom{\hone}{\mbox{\hspace{1em}}}
\ncom{\htwo}{\mbox{\hspace{2em}}}
\ncom{\hthree}{\mbox{\hspace{3em}}}
\ncom{\hfour}{\mbox{\hspace{4em}}}
\ncom{\hsev}{\mbox{\hspace{7em}}}
\ncom{\vone}{\vskip 2ex}
\ncom{\vtwo}{\vskip 4ex}
\ncom{\vonee}{\vskip 1.5ex}
\ncom{\vthree}{\vskip 6ex}
\ncom{\vfour}{\vspace*{8ex}}
\ncom{\norm}{\|\;\;\|}
\ncom{\integ}[4]{\int_{#1}^{#2}\,{#3}\,d{#4}}
\ncom{\inp}[2]{\langle{#1},\,{#2} \rangle}
\ncom{\Inp}[2]{\Langle{#1},\,{#2} \Langle}
\ncom{\vspan}[1]{{{\rm\,span}\#1 \}}}
\ncom{\dm}[1]{\displaystyle {#1}}

\begin{document}
\title[Weyl's Theorem  in Two Variables]
{Weyl's Theorem for Pairs of Commuting Hyponormal Operators}
\author[S. Chavan and R. Curto]{Sameer Chavan and Ra$\acute{\mbox{u}}$l Curto}
\address{Indian Institute of Technology Kanpur\\
Kanpur- 208016, India}
\email{chavan@iitk.ac.in}
\address{University of Iowa\\
Iowa City, Iowa 52242, USA}
\email{raul-curto@uiowa.edu}
\thanks{The second named author was partially supported by NSF Grant DMS-1302666.} \ 



\begin{abstract}
Let $\mathbf{T}$ be a pair of commuting hyponormal operators satisfying the so-called quasitriangular property
$$
\textrm{dim} \; \textrm{ker} \; (\mathbf{T}-\boldsymbol\lambda) \ge \textrm{dim} \; \textrm{ker} \; (\mathbf{T} - {\boldsymbol\lambda})^*),
$$
for every $\boldsymbol\lambda$ in the Taylor spectrum $\sigma(\mathbf{T})$ of $\mathbf{T}$.
We prove that the Weyl spectrum of $\mathbf{T}$, $\omega(\mathbf{T})$, satisfies the identity
$$
\omega(\mathbf{T})=\sigma(\mathbf{T}) \setminus \pi_{00}(\mathbf{T}),
$$
where $\pi_{00}(\mathbf{T})$ denotes the set of isolated eigenvalues of finite multiplicity.

Our method of proof relies on a (strictly $2$-variable) fact about the topological boundary of the Taylor spectrum; as a result, our proof does not hold for $d$-tuples of commuting hyponormal operators with $d>2$. 
\end{abstract}

\maketitle

\section{Weyl's Theorem in Two Variables}

The aim of this note is to present an analog of Weyl's Theorem for commuting pairs of hyponormal operators.
For the definitions and basic theory of various spectra
including the Taylor and Harte spectra, the reader is referred to \cite{Cu} (see also \cite{Tay1}). \
For a commuting $d$-tuple $\mathbf{T}$, we reserve the symbols $\sigma(\mathbf{T})$, $\sigma_H(\mathbf{T})$, $\sigma_p(\mathbf{T})$ and $\sigma_e(\mathbf{T})$ for the Taylor spectrum,  Harte spectrum, point spectrum and Taylor essential spectrum of $\mathbf{T}$, respectively. \
By a {\it commuting $d$-tuple} $\mathbf{T}$, we understand here and throughout this note a $d$-tuple of commuting bounded linear operators $T_1, \cdots, T_d$ on a complex, separable Hilbert space $\mathcal H$. \ 
For $d=1$, L. Coburn proved in \cite{C} Weyl's Theorem for hyponormal operators; this led to a number of extensions to classes of operators containing the subnormal operators. \ For $d>1$, there are various notions of Weyl spectrum (\cite{Cho-1}, \cite{H}, \cite{P}, \cite{L}). \ We recall in particular two notions of Weyl spectrum with which we will be primarily concerned. \
The {\it joint Weyl spectrum} $\omega(\mathbf{T})$ of a commuting $d$-tuple $\mathbf{T}$ is defined as
$$
\omega(\mathbf{T}):=\cap \{\sigma(\mathbf{T}+\mathbf{K}) : \mathbf{K} \in \mathcal K^{(d)}(\mathcal H)~\mbox{~such~that~} \mathbf{T}+\mathbf{K}~\mbox{~is~commuting}\},
$$
where $\mathcal \mathbf{K}^{(d)}(\mathcal H)$ denotes the collection of $d$-tuples of compact operators on $\mathcal H$. \ The {\it Taylor Weyl spectrum} of $\mathbf{T}$ is defined as
$$
\sigma_W(\mathbf{T}):=\sigma_e(\mathbf{T}) \cup \{\boldsymbol\lambda \in \sigma(\mathbf{T}) \setminus \sigma_e(\mathbf{T}) : \mbox{ind}(\mathbf{T}-\boldsymbol\lambda) \neq 0\},
$$
where $\boldsymbol\lambda:=(\lambda_1,\cdots,\lambda_d)$ and the {\it Fredholm index} $\mbox{ind}(\mathbf{S})$ of a $d$-tuple $\mathbf{S}$ of commuting operators is the Euler characteristic of the Koszul complex $K(\mathbf{S}, \mathcal{H})$ for $\mathbf{S}$, given by
\beq
\label{index} \mbox{ind}(\mathbf{S}):= \sum_{k=0}^d (-1)^k \dim H^k(\mathbf{S}),
\eeq
with $H^k(\mathbf{S})$ denoting the $k$-th cohomology group in $K(\mathbf{S},\mathcal{H})$; observe that $H^0(\mathbf{S}) \cong \ker \; Q_{\mathbf{S}}(I)$ and $H^d(\mathbf{S}) \cong \ker \; Q_{\mathbf{S^*}}(I)$. \ By {\it Weyl spectrum} we understand {\it any} of the joint Weyl and Taylor Weyl spectra. \  

For future reference, we record the following elementary fact from \cite{H-K} about the relationship between the aforementioned two notions of Weyl spectra. \ For the sake of completeness, we provide an alternative verification of this result.

\begin{lemma} (\cite[Lemma 2]{H-K}) \label{inclu}
The joint Weyl spectrum and Taylor Weyl spectrum of a commuting $d$-tuple $\mathbf{T}$ satisfies the relation $\sigma_W(\mathbf{T}) \subseteq  \omega(\mathbf{T})$.
\end{lemma}

\begin{proof}
Suppose that there exists a $d$-tuple $\mathbf{K}$ of compact operators such that $\mathbf{T}+\mathbf{K}$ is a commuting $d$-tuple and $\boldsymbol\lambda \notin \sigma(\mathbf{T}+\mathbf{K})$. \ Thus, $\mathbf{T}-\boldsymbol\lambda + \mathbf{K}$ is invertible, and hence Fredholm with Fredholm index equal to $0$. \ By the multivariable Atkinson Theorem \cite[Theorem 2]{Cu-0}, $\mathbf{T}-\boldsymbol\lambda$ is Fredholm with index equal to $0$, that is, $\boldsymbol\lambda \notin \sigma_W(\mathbf{T}).$
\end{proof}

\begin{remark} \ It is interesting to note that $\sigma_W(\mathbf{T}) = \omega(\mathbf{T})$ for all $\mathbf{T}$ if and only if $d=1$. \ This may be concluded from the discussion following \cite[Theorem]{R-2}, where it is shown that certain Fredholm $d$-tuples of index equal to $0$ cannot be perturbed by compact tuples to a Taylor invertible tuple (see also \cite{H-K}). \qed
\end{remark}

Before we present an analog of Weyl's Theorem \cite[Theorem 3.1]{C} for commuting hyponormal tuples,
recall that a bounded, linear operator $S$ on a Hilbert space $\mathcal H$ is {\it hyponormal} if its self-commutator $S^*S - SS^*$ is positive. \ Also, given a $d$-tuple $\mathbf{S} \equiv (S_1,\cdots,S_d)$ we let 
$$
Q_{\mathbf{S}}(X):=\sum_{i=1}^d S^*_iXS_i \; \; (\textrm{for } X \textrm{ a bounded operator on } \mathcal{H}). 
$$

\begin{definition}
(cf. \cite{H-K}). \ A $d$-tuple $\mathbf{T}$ has the {\it quasitriangular property} if $\mathbf{T}$ satisfies 
\beq  \label{qt} 
\dim \ker Q_{\mathbf{T}_{\boldsymbol\lambda}}(I) \geq  \dim \ker Q_{\mathbf{T}^*_{\boldsymbol\lambda}}(I),
\eeq
for every $\boldsymbol\lambda \in \sigma(\mathbf{T})$, where $\mathbf{T}_{\boldsymbol\lambda}:=\mathbf{T}-\boldsymbol\lambda$.
\end{definition}
We are now ready to state our main result.
 
\begin{theorem} \label{Weyl} Let $\mathbf{T}$ be a commuting $d$-tuple of hyponormal operators on a Hilbert space $\mathcal H$ and let
$\pi_{00}(\mathbf{T})$ denote the set of isolated eigenvalues of $\mathbf{T}$ of finite multiplicity. \ The following statements are true. \newline
(i) \ $\omega(\mathbf{T}) \subseteq \sigma(\mathbf{T}) \setminus \pi_{00}(\mathbf{T})$. \ \newline
(ii) \ Assume $d=2$ and that $\mathbf{T}$ satisfies (\ref{qt}). \ Then $\sigma(\mathbf{T}) \setminus \pi_{00}(\mathbf{T}) \subseteq \sigma_W(\mathbf{T})$.
\end{theorem}

\begin{remark}
Condition (\ref{qt}) is equivalent to the statement that the dimension of the cohomology group at the first stage of the Koszul complex for $\mathbf{T}-\boldsymbol\lambda$ is greater than or equal to the dimension of the cohomology group at the last stage of the Koszul complex for $\mathbf{T}-\boldsymbol\lambda$. \ This property is closely related to the notion of quasitriangular operator (see the discussion following \cite[Definition 3]{H-K}). \ Finally, note that \eqref{qt} is satisfied by any $d$-tuple $\mathbf{T}$ such that $\sigma_p(T_i^*)=\emptyset$ for some $i=1,\cdots,d$. \qed
\end{remark}

The following is immediate from Theorem \ref{Weyl} and Lemma \ref{inclu}.

\begin{corollary} \label{Weyl-0} Let $\mathbf{T}$ be a commuting pair of hyponormal operators on a Hilbert space $\mathcal H$ and let
$\pi_{00}(\mathbf{T})$ denotes the set of isolated eigenvalues of $\mathbf{T}$ of finite multiplicity.
If $\mathbf{T}$ satisfies the quasitriangular property \eqref{qt} then
\begin{equation} \label{Cor15}
\omega(\mathbf{T}) = \sigma(\mathbf{T}) \setminus \pi_{00}(\mathbf{T}) = \sigma_W(\mathbf{T}).
\end{equation}
\end{corollary}

\begin{proof}
By Theorem \ref{Weyl}(i), 
$$
\omega(\mathbf{T}) \subseteq \sigma(\mathbf{T}) \setminus \pi_{00}(\mathbf{T}),
$$
and by Theorem \ref{Weyl}(ii),
$$
\sigma(\mathbf{T}) \setminus \pi_{00}(\mathbf{T}) \subseteq \sigma_W(\mathbf{T}).
$$
Since $\sigma_W(\mathbf{T}) \subseteq \omega(\mathbf{T})$ is always true (by Lemma \ref{inclu}), (\ref{Cor15}) follows. 
\end{proof}

\begin{remark} \ (i) One cannot relax the condition \eqref{qt}. \ Indeed, let $\mathbf{T}=(U_+, 0)$, where $U_+$ denotes the unilateral shift on $\ell^2(\mathbb Z_+)$ (see the discussion following \cite[Definition 5]{H-K}). \ Indeed, for this commuting pair, $\sigma(\mathbf{T})=\bar{\mathbb D} \times {0}$, $\sigma_e(\mathbf{T})=\mathbb{T} \times {0}$, $\pi_{00}(\mathbf{T})=\emptyset$, $\omega(\mathbf{T})=\sigma(\mathbf{T})$ and $\sigma_W(\mathbf{T})=\sigma_e(\mathbf{T})$. \ \newline
(ii) \ Note further that the result above is not best possible. \ Indeed, it may happen that the conclusion of Weyl's Theorem holds but \eqref{qt} is violated. \ For instance, let $\mathbf{T}$ be the Drury-Arveson $2$-variable weighted shift; then $T_1$ and $T_2$ are hyponormal operators such that $\pi_{00}(\mathbf{T})=\emptyset$ and $\omega(\mathbf{T}) = \sigma(\mathbf{T}) = \sigma_W(\mathbf{T})=\overline{\mathbb B}$  \cite{Ar3}. \ However, since $I - Q_\mathbf{T}(I) \leq 0$ and $I-Q_{\mathbf{T}^*}(I)$ is the orthogonal projection onto the scalars, 
$$
0=\dim \ker Q_{\mathbf{T}}(I) \ngeq  \dim \ker Q_{\mathbf{T}^*_{\boldsymbol\lambda}}(I)=1 \; \; \textrm{ for every } \boldsymbol\lambda \in \sigma(\mathbf{T}). \qed
$$
\end{remark}

The first part of Theorem \ref{Weyl} generalizes \cite[Theorem 4]{Cho-1}, while Corollary \ref{Weyl-0} generalizes \cite[Theorem 6]{H-K} when $d=2$ (see also \cite[Theorem 2.5.4]{Le}). \ All these were obtained under the additional assumption that $\mathbf{T}$ is {\it doubly commuting}, that is, a commuting $d$-tuple $\mathbf{T}$ such that 
$$
T^*_iT_j=T_jT^*_i~\mbox{for~} 1 \leq i \neq j \leq d.
$$
Although  the conclusion of  \cite[Theorem 6]{H-K} is stronger than that of Theorem \ref{Weyl}, our result does not assume double commutativity. \ On the other hand, our method of proof relies on a (consequence of a) strictly $2$-dimensional result about the topological boundary of the Taylor spectrum \cite[Theorem 6.8]{Cu}, and hence does not extend to the case $d \geq 3.$

\section{Proof of Theorem \ref{Weyl}}

The proof of Theorem \ref{Weyl} presented below relies on a number of non-trivial results, including the Shilov Idempotent Theorem (\cite{Cu}, \cite{Tay2}). \ We also need a result from \cite{Cu-1} pertaining to
connections between Harte and Taylor spectra.
We begin with a decomposition of tuples of commuting hyponormal operators; we believe this result is known, although we have not been able to find a concrete reference in the literature. \ In what follows, recall that $\boldsymbol\lambda := (\lambda_1, \cdots, \lambda_d)$.  

\begin{lemma} \label{Lem} Let $\mathbf{T}$ be a $d$-tuple of commuting
hyponormal operators $T_1, \cdots, T_d$. \ Then for any $\boldsymbol\lambda \equiv (\lambda_1,
\cdots, \lambda_d) \in \sigma(\mathbf{T}),$ $\mathcal M_1:=\ker
Q_{\mathbf{T}_{\boldsymbol\lambda}}(I)$ is a reducing subspace for $\mathbf{T}$, where
$\mathbf{T}_{\boldsymbol\lambda}$ denotes the $d$-tuple $\mathbf{T} - \boldsymbol\lambda I=(T_1-\lambda_1 I,
\cdots, T_d - \lambda_d I)$, and $Q_{\mathbf{T}_{\boldsymbol\lambda}}(I):=\sum_{i=1}^d
(T_i-\lambda_i I)^*(T_i-\lambda_i I)$. \ Moreover, $\mathbf{T}$ decomposes into
$(\lambda_1 I_{\mathcal M_1},\cdots, \newline \lambda_d I_{\mathcal M_1}) \oplus \mathbf{B}$ on the orthogonal direct sum
$\mathcal H = \mathcal M_1 \oplus (\mathcal M_1)^{\perp},$ where
$I_{\mathcal M_1}$ is the identity operator on $\mathcal M_1,$ and
$\mathbf{B}$ is a $d$-tuple of commuting hyponormal operators such that $\ker
Q_{\mathbf{B}_{\boldsymbol\lambda}}(I_{\mathcal M^{\perp}_1})=\{0\}$.
\end{lemma}

\begin{proof} Note that $\mathcal M_1=\ker Q_{\mathbf{T}_{\boldsymbol\lambda}}(I) = \cap_{i=1}^d \ker (T_i - \lambda_i I).$
Clearly, $T_i (\mathcal M_1) \subseteq \mathcal M_1$ for any
$i=1, \cdots, d$. \ Since $T_i$ is hyponormal, $T_ix=\lambda_i x$
implies $T^*_ix = \bar{\lambda}_ix,$ and hence $T^*_i (\mathcal M_1)
\subseteq \mathcal M_1$ for any $i=1, \cdots, d$. \ Since $\mathbf{T}$ is an
extension of $\mathbf{B}$, it follows immediately that $\mathbf{B}$ has the desired
properties.
\end{proof}

To state the next result, we recall that a pair $\mathbf{S}$ of commuting operators is said to be semi-Fredholm if all boundary maps in the Koszul complex $K(\mathbf{S},\mathcal{H})$ have closed range, and either $H^0(\mathbf{S})$ and $H^2(\mathbf{S})$ are finite dimensional or $H^1(\mathbf{S})$ is finite dimensional.
   
\begin{lemma} \label{Cor610}
\ (cf. \cite[Corollary 3.6]{Cu-1} and \cite[Corollary 6.10]{Cu}) \ Let $\mathbf{T}$ be a pair of commuting operators, and let $\boldsymbol\lambda$ be an isolated point of $\sigma_H(\mathbf{T})$. \ Assume that $\boldsymbol\lambda$ is in the semi-Fredholm domain of $\mathbf{T}$. \ Then $\boldsymbol\lambda$ is an isolated point of $\sigma(\mathbf{T})$ if and only if $\mbox{ind} (\mathbf{T}-\boldsymbol\lambda) = 0$. \  
\end{lemma} 

\begin{proof}[Proof of Theorem \ref{Weyl}]
(i) \ To see the inclusion $ \omega(\mathbf{T}) \subseteq \sigma(\mathbf{T}) \setminus
\pi_{00}(\mathbf{T})$, it suffices to check that any isolated eigenvalue of
$\mathbf{T}$ of finite multiplicity does not belong to the Taylor spectrum of
some finite rank perturbation of $\mathbf{T}$. \ To see that, let
$\boldsymbol\lambda=(\lambda_1, \cdots, \lambda_d)$ be an isolated point of the
Taylor spectrum of $\mathbf{T}=(T_1, \cdots, T_d).$
Let $K_1:=\{\boldsymbol\lambda\}$ and let $K_2:=\sigma(\mathbf{T}) \setminus K_1$. \ By the
Shilov Idempotent Theorem \cite[Application 5.24]{Cu}, there exist
invariant subspaces $\mathcal M_1, \mathcal M_2$ of $\mathbf{T}$ such that
$\mathcal H = \mathcal M_1 \dotplus \mathcal M_2$ (Banach direct sum) and
$\sigma(\mathbf{T}|_{M_i}) =K_i$ for $i=1, 2$. \ By the Spectral Mapping Property
for the Taylor spectrum \cite[Corollary 3.5]{Cu}, $\sigma(\mathbf{T}|_{\mathcal{M}_1} - \boldsymbol\lambda
I_{\mathcal M_1})=\{0\}$,  where $I_{\mathcal M_1}$ is the identity
operator on $\mathcal M_1$. \ \ Since $\mathbf{T}|_{\mathcal{M}_1} - \boldsymbol\lambda I_{\mathcal
M_1}$ is a commuting $d$-tuple of hyponormal operators, it follows from the
Projection Property for the Taylor spectrum \cite[Theorem 4.9]{Cu} that
$T_{i}|_{\mathcal{M}_1} = \lambda_i I_{\mathcal M_1}$ for $i=1, \cdots, d$. \ In
particular, $\mathcal M_1 \subseteq \cap_{i=1}^d \ker (T_i -
\lambda_i I_{\mathcal M_1})$. \ Since $\boldsymbol\lambda \notin K_2,$ we must
have $\mathcal M_1=\cap_{i=1}^d \ker (T_i - \lambda_i I_{\mathcal
M_1})$. \ 
By the preceding lemma, $\mathbf{T}$ decomposes into $\boldsymbol\lambda I \oplus \mathbf{B}$ on the orthogonal direct sum $\mathcal H = \mathcal M_1 \oplus \mathcal M^{\perp}_1,$ where $\mathbf{B}$ is a commuting $d$-tuple of  hyponormal operators obtained by restricting $\mathbf{T}$ to
$\mathcal M^{\perp}_1$. \ 

Suppose now that $\boldsymbol \lambda \in \sigma(\mathbf{B})$. \ Since $\sigma(\mathbf{T}) =
\{\boldsymbol\lambda\} \cup \sigma(\mathbf{B})$ \cite[Page 39]{Cu}, $\boldsymbol\lambda$ is an
isolated point of $\sigma(\mathbf{B})$. \ Since $\mathbf{B}=(B_1, \cdots, B_d)$ is a
$d$-tuple of hyponormal operators, by the argument of the preceding
paragraph, isolated points of $\sigma(\mathbf{B})$ must be eigenvalues of
$\mathbf{B}$, and hence there exists $0 \ne y \in \mathcal{M}_1^{\perp}$ such that $y \in \cap_{i=1}^d \ker (B_i - \lambda_i I_{\mathcal M_1^{\perp}})$. \ It follows that $y \in \cap_{i=1}^d \ker (T_i - \lambda_i I_{\mathcal{H}})=\mathcal M_1$, which is a contradiction. \ Thus, $\boldsymbol\lambda \notin \sigma(\mathbf{B})$. \ 

We have proved that if $\boldsymbol\lambda$ is an isolated eigenvalue of $\mathbf{T}$ having finite multiplicity then $\mathbf{B}=\mathbf{T} - \boldsymbol\lambda
I_{\mathcal M_1}$ is a finite-rank perturbation of $\mathbf{T}$, and $\boldsymbol\lambda \notin \sigma(\mathbf{B})$, as desired.

(ii) Assume now that $d=2$ and that $\mathbf{T}$ has the
quasitriangular property \eqref{qt}. \ Since each $T_i$ is
hyponormal, \beq \label{qtp} \dim \ker Q_{\mathbf{T}_{\boldsymbol\lambda}}(I) =  \dim
\ker Q_{\mathbf{T}^*_{\boldsymbol\lambda}}(I)~\mbox{for~ every~} \boldsymbol\lambda \in
\sigma(\mathbf{T}).\eeq Let $\boldsymbol\lambda \in \sigma(\mathbf{T})$ be such that $\mathbf{T}-\boldsymbol\lambda$
is Fredholm with Fredholm index equal to $0$. \ By the definition of
the Fredholm index (see \eqref{index}) and \eqref{qtp},
\beq
\label{dim-cohom} 2 \dim \ker Q_{\mathbf{T}_{\boldsymbol\lambda}}(I) = \dim
H^1(\mathbf{T}_{\boldsymbol\lambda}) = 2\dim \ker Q_{\mathbf{T}^*_{\boldsymbol\lambda}}(I), 
\eeq
where $H^1(\mathbf{S})$ is the middle cohomology group appearing in  the Koszul
complex of the commuting pair $\mathbf{S}$. \ Since $\boldsymbol\lambda \in \sigma(\mathbf{T}) \setminus
\sigma_e(\mathbf{T})$, by \eqref{dim-cohom} we must necessarily have $0 < \dim \ker
Q_{\mathbf{T}_{\boldsymbol\lambda}}(I) < \infty,$ and hence $\boldsymbol\lambda$ is an eigenvalue
of $\mathbf{T}$ of finite multiplicity. \ 

To see that $\boldsymbol\lambda$ is an isolated
point of $\sigma(\mathbf{T})$, in view of Lemma \ref{Cor610}, it
suffices to check that $\boldsymbol\lambda$ is an isolated point of the Harte
spectrum $\sigma_H(\mathbf{T})$ of $\mathbf{T}$. \ If $\mathbf{B}$ is as in Lemma \ref{Lem}, then
$\mathbf{B}-\boldsymbol\lambda I_{\mathcal M^{\perp}_1}$ is also Fredholm with index
equal to $0$. \ Since $\ker Q_{\mathbf{B}_{\boldsymbol\lambda}}(I_{\mathcal M^{\perp}_1})
=\{0\}$ (by Lemma \ref{Lem}), using \eqref{qtp} we must have $\ker
Q_{\mathbf{B}^*_{{\boldsymbol\lambda}}}(I_{\mathcal M^{\perp}_1}) =\{0\}$. \ 

On the other hand, $Q_{\mathbf{B}_{\boldsymbol\lambda}}(I_{\mathcal
M^{\perp}_1})$ and $Q_{\mathbf{B}^*_{\boldsymbol\lambda}}(I_{\mathcal M^{\perp}_1})$ are
Fredholm (by item (v) in the paragraph immediately following \cite[Remark 6.7]{Cu}). \ As a result, $Q_{\mathbf{B}_{\boldsymbol\lambda}}(I_{\mathcal M^{\perp}_1})$ and
$Q_{\mathbf{B}^*_{\boldsymbol\lambda}}(I_{\mathcal M^{\perp}_1})$ are invertible. \ It
follows that $\boldsymbol\lambda$ cannot be in the Harte spectrum
$\sigma_H(\mathbf{B})$. \ Since $\boldsymbol\lambda \in \sigma_H(\mathbf{T}) = \{\boldsymbol\lambda\} \cup
\sigma_H(\mathbf{B})$, $\boldsymbol\lambda$ must be an isolated point of $\sigma_H(\mathbf{T})$, as desired.
\end{proof}

\begin{remark} \label{rem1}
Assume that the commuting pair $\mathbf{T}$ satisfies \eqref{qt}. \ 
If $\mathbf{T}$ has no normal direct summand, then by Lemma \ref{Lem}, $\mathbf{T}$ has no eigenvalues, and hence the Weyl spectrum of $\mathbf{T}$ coincides with the Taylor spectrum. \qed
\end{remark}

\section{Some Consequences of Theorem \ref{Weyl}}

In this section we discuss a couple of interesting consequences of Theorem \ref{Weyl} (cf. \cite[Corollary 3.2]{C}). \ First, we recall the notion of jointly hyponormality for $d$-tuples. \ A $d$-tuple $\mathbf{T}=(T_1, \cdots, T_d)$ of bounded linear operators on ${\mathcal H}$ is said to be {\it jointly hyponormal}
if
the $d
\times d$ matrix $([T^*_j, T_i])_{1 \leq i, j \leq d}$ is positive
definite, where $[A, B]$ stands for the commutator $AB-BA$ of $A$
and $B$. \ 
 
\begin{corollary}  Let $\mathbf{T}=(T_1, T_2)$ be a jointly hyponormal commuting pair with the quasitriangular property \eqref{qt}.
If $\mathbf{T}$ has no isolated eigenvalues of finite multiplicity, then for any pair $\mathbf{K}=(K_1, K_2)$ of compact operators $K_1, K_2$ such that $\mathbf{T}+\mathbf{K}$ is commuting, we have
\beqn
\|T^*_1T_1 + T^*_2T_2\| \leq \|(T_1+K_1)^*(T_1 +K_1) + (T_2+K_2)^*(T_2 +K_2)\|.
\eeqn
\end{corollary}

\begin{proof} Since $\mathbf{T}$ has no isolated eigenvalues of finite multiplicity, Theorem \ref{Weyl} implies that 
$$
\sigma(\mathbf{T}) \subseteq \sigma(\mathbf{T}+\mathbf{K})
$$
for any pair $\mathbf{K}$ of compact operators such that $\mathbf{T}+\mathbf{K}$ is
commuting. \ Now apply \cite[Lemma 3.10]{CS} and \cite[Lemma 2.1]{Ch-4} to conclude
that 
\beqn \|T^*_1T_1 + T^*_2T_2\| &=& r(\mathbf{T})^2 \leq r(\mathbf{T}+\mathbf{K})^2
\\ &\leq & \|(T_1+K_1)^*(T_1 +K_1) + (T_2+K_2)^*(T_2 +K_2)\|,
\eeqn
where 
$$
r(\mathbf{S}):=\sup \big\{\sqrt{|z_1|^2 + \cdots + |z_d|^2} : (z_1,
\cdots, z_d) \in \sigma(S)\big\}
$$
denotes the geometric spectral radius
of the $d$-tuple $\mathbf{S}$ of bounded linear operators on $\mathcal H$. \ 
\end{proof}

We do not know whether the conclusion of the last corollary holds for commuting pairs of hyponormal operators satisfying the quasitriangular property. \ (Recall that there exist commuting pairs of subnormal operators which are not jointly hyponormal; the Drury-Arveson $2$-variable weighted shift is such an example.)

Next, we obtain an analog of the Riesz-Schauder Theorem for commuting pairs of hyponormal operators (cf. \cite[Corollary 2.5.6]{Le}). \ Recall that a commuting $d$-tuple is {\it Browder invertible} if $\mathbf{T}$ is Fredholm such that there exists a deleted neighborhood of $\mathbf{0}$ disjoint from the Taylor spectrum of $\mathbf{T}$. \ The {\it Browder spectrum} $\sigma_b(\mathbf{T})$ of $\mathbf{T}$ is the collection of those $\boldsymbol\lambda \in \mathbb C^d$ for which $\mathbf{T}-\boldsymbol\lambda$ is not Browder invertible. \ It is not hard to see that $\sigma_b(\mathbf{T})$ is the union of $\sigma_e(\mathbf{T})$ and the accumulation points of the Taylor spectrum $\sigma(\mathbf{T})$ (cf. \cite{CuD}). (For basic facts about the Browder spectrum in the case $d=1$, the reader is referred to \cite{BDW}.) 

\begin{corollary}
Let $\mathbf{T}=(T_1, T_2)$ be a jointly hyponormal $2$-tuple with the quasitriangular property \eqref{qt}. \ Then $\sigma_b(\mathbf{T})=\omega(\mathbf{T}).$
\end{corollary}

\begin{proof}
The inclusion $\omega(\mathbf{T}) \subseteq \sigma_b(\mathbf{T})$ is always true \cite[Lemma 2.5.3]{Le}.
To see the reverse inclusion, let $\boldsymbol\lambda \notin \omega(\mathbf{T})$. \ By Lemma \ref{inclu},
$\boldsymbol\lambda \notin \sigma_W(\mathbf{T})$. \ In particular, $\boldsymbol\lambda \notin \sigma_e(\mathbf{T}).$
By Theorem \ref{Weyl}(ii), if $\boldsymbol\lambda \notin \sigma_W(\mathbf{T})$ then $\boldsymbol\lambda$ is an isolated point of $\sigma(\mathbf{T})$. \ Since $\sigma_b(\mathbf{T})$ is the union of $\sigma_e(\mathbf{T})$ and the accumulation points of $\sigma(\mathbf{T})$, it follows that $\boldsymbol\lambda \notin \sigma_b(\mathbf{T})$. \ The desired inclusion is now clear. 
\end{proof}

\end{document}